\documentclass[11pt,a4paper]{article}
\usepackage{amsmath}
\usepackage{tikz}
\usepackage{mathdots}
\usepackage{yhmath}
\usepackage{cancel}
\usepackage{color}
\usepackage{siunitx}
\usepackage{array}
\usepackage{multirow}
\usepackage{amssymb}
\usepackage{gensymb}
\usepackage{tabularx}
\usepackage{extarrows}
\usepackage{booktabs}
\usepackage[colorlinks=true,linkcolor=blue,citecolor=blue,urlcolor=blue, ]{hyperref}
\usepackage[doipre={doi:~}]{uri}
\usepackage{amsmath,amssymb,amsthm,amsfonts,latexsym,graphicx,subfigure}
\usepackage[all,import]{xy}
\usepackage[sort&compress, numbers]{natbib}
\usepackage{indentfirst}
\usepackage{fancyhdr,enumerate}
\usepackage{bbm,color}
\usepackage{tikz}
\usetikzlibrary{shapes.geometric, arrows}
\usepackage{accents,cases}
\usepackage{multirow}
\usepackage[lined, ruled, linesnumbered, longend]{algorithm2e}
\usepackage{natbib}  
\usepackage{hyperref}  
\hypersetup{colorlinks=true, urlcolor=blue} 

\usepackage{array,float}
\newcommand{\PreserveBackslash}[1]{\let\temp=\\#1\let\\=\temp}
\newcolumntype{C}[1]{>{\PreserveBackslash\centering}p{#1}}
\newcolumntype{R}[1]{>{\PreserveBackslash\raggedleft}p{#1}}
\newcolumntype{L}[1]{>{\PreserveBackslash\raggedright}p{#1}}
\newcommand{\Mylim}[2]{\mathop{\lim}\limits_{#1}{#2}}

\makeatletter
\def\wbar{\accentset{{\cc@style\underline{\mskip8mu}}}}

\makeatother

\newcommand{\kyuta}[1]{\mathop{#1}\limits^\sim}

\theoremstyle{plain}

\newtheorem{theorem}{Theorem}[section] 
\newtheorem{definition}[theorem]{Definition} 
\newtheorem{lemma}[theorem]{Lemma}
\newtheorem{corollary}[theorem]{Corollary}
\newtheorem{example}[theorem]{Example}
\newtheorem{proposition}[theorem]{Proposition}
\newtheorem{condition}[theorem]{Condition}
\newtheorem{remark}[theorem]{Remark}

\numberwithin{equation}{section}

\newcounter{FigureCounter}
\usetikzlibrary{fadings}
\usetikzlibrary{patterns}
\usetikzlibrary{shadows.blur}
\usetikzlibrary{shapes}
\allowdisplaybreaks

\begin{document}

\title{On finiteness of spectral radius order}
\author{Yanlong Ding\footnotemark[1]\and Chuanyuan Ge\footnotemark[2] \and Shiping Liu\footnotemark[3]}

\footnotetext[1]{School of Mathematical Sciences, 
University of Science and Technology of China, Hefei 230026, China. \\
Email address:
{\tt dylustc@mail.ustc.edu.cn} }
\footnotetext[2]{School of Mathematics and Statistics, Fuzhou University, Fuzhou 350108, China. \\
Email address:
{\tt gechuanyuan@mail.ustc.edu.cn} }
\footnotetext[3]{School of Mathematical Sciences, 
University of Science and Technology of China, Hefei 230026, China. \\
Email address:
{\tt spliu@ustc.edu.cn}
}

\date{}\maketitle

\begin{abstract}
The concept of spectral radius order plays an crucial role in the breakthrough work on equiangular lines due to Jiang, Tidor, Yao, Zhang, and Zhao [Ann. of Math. (2) 194 (2021), no. 3, 729-743]. However, it is difficult to calculate the spectral radius order explicitly in general, or even to characterize numbers with finite spectral radius order. In this paper, we characterize numbers with finite spectral radius orders in two special classes: quadratic algebraic integers and the numbers no larger than $2$. Additionally, we derive precise values of the spectral radius order of two infinite families of quadratic algebraic integers.

\end{abstract}

\section{Introduction}\label{sec:intro}
{The spectral radius order is a very natural concept introduced by Jiang and Polyanskii \cite{Jiang-Polyanskii-20}. Their motivation originates from the study of equiangular lines. 
\begin{definition}\label{def:kappa}
       For any $\lambda\in (0,\infty)$, the spectral radius order $\kappa(\lambda)$ of $\lambda$ is defined as
       $$
\kappa(\lambda)=\min\left\{ m\ \vert\  \text{there exists a graph } G=(V,E) \mbox{ with }|V|=m \mbox{ and } \lambda_1(G)=\lambda\right\},
        $$
 where $\lambda_1(G)$ is the spectral radius of $G$, that is, the largest eigenvalue of the adjacency matrix of $G$. If there exists no graph with spectral radius $\lambda$, we set  $\kappa(\lambda)=+\infty$. 
\end{definition}
 A set of lines passing through the origin in $\mathbb{R}^d$ is called equiangular if the angle between any two of them is constant. Lemmens and Seidel \cite{Lemmens-Seidel-73} proposed to determine the maximum number $N_\alpha(d)$ of equiangular lines in $\mathbb{R}^d$ with a fixed angle $\arccos\alpha$. In fact, they completely determined the values of $N_{1/3}(d)$ for all $d$ and proved particularly $N_{1/3}(d)=2(d-1)$ for all sufficiently large $d$. Neumaier \cite{Neumaier-89} subsequently proved that $N_{1/5}(d)=\lfloor 3(d-1)/2\rfloor$ for all sufficiently large $d$. Cao, Koolen, Lin, and Yu \cite{Cao-Koolen-Lin-Yu} proved that $N_{1/5}(d)=\lfloor 3(d-1)/2\rfloor$ for $d\geq 186$. For $k\geq 3$, Bukh \cite{Bukh-16} conjectured that $N_{1/(2k-1)}(d)=kd/(k-1)+O_k(1)$ as $d\rightarrow\infty.$

Building upon previous works on general $\alpha$ \cite{Bukh-16,MR3742757},  Jiang and Polyanskii \cite{Jiang-Polyanskii-20} conjectured that for any $\alpha\in (0,1)$
\[\lim_{d\rightarrow\infty}\frac{N_\alpha(d)}{d}=\frac{\kappa(\lambda) }{\kappa(\lambda)-1}, \,\,\text{where}\,\,\lambda:=\frac{1-\alpha}{2\alpha}.\]
Jiang and Polyanskii \cite{Jiang-Polyanskii-20} themselves confirmed the conjecture when $\lambda<\sqrt{2+\sqrt{5}}$. Since $\kappa(\lambda)=\lambda+1$ for any $\lambda\in \mathbb{Z}^+$, Jiang and Polyanskii's conjecture covers Bukh's conjecture.
   
In 2019,  Jiang, Tidor, Yao, Zhang, and Zhao \cite{MR4334975} confirmed Jiang and Polyanskii's conjecture. Exactly, they proved the following theorem. 
\begin{theorem}\label{thm:k-lambad}
    Fix $\alpha\in (0,1)$. Set $\lambda:=\frac{1-\alpha}{2\alpha}$. We have
\begin{itemize}
    \item [(i)]  $N_{\alpha}(d)=\left\lfloor \frac{(d-1)\kappa(\lambda)}{\kappa(\lambda)-1}\right\rfloor$ for $d>2^{2^{C\kappa(\lambda)/\alpha}}$ when $\kappa(\lambda)<\infty$, where $C$ is an absolute constant;
    \item  [(ii)]  $N_{\alpha}(d)=d+o_{\alpha}(d)$ as $d\to \infty$ when $\kappa(\lambda)=\infty$.
\end{itemize}
\end{theorem}
 A natural problem arising from their results is to determine all \( \lambda \) for which \( \kappa(\lambda) \) is finite and calculate the spectral radius order of $\lambda$ explicitly. 
 

For a graph, all its eigenvalues are totally real algebraic integers. By the Perron-Frobenius theorem, the largest eigenvalue is greater than or equal to the absolute value of the other eigenvalues. In particular, the largest eigenvalue is the spectral radius. Thus, if \( \kappa(\lambda) < \infty \), then \( \lambda \) must satisfy the following two conditions.
\begin{condition}\label{condition}
    $ $
    \begin{enumerate}
    \item[(1)]\label{conditionsA} { \( \lambda \) is a totally real algebraic integer; }
    \item[(2)]\label{conditionsC}  { \( \lambda \) is greater than or equal to the absolute value of its algebraic conjugates.}
\end{enumerate}
\end{condition}

In general, Condition \ref{condition} is not sufficient to ensure the finiteness of the spectral radius order, see \cite[Remake 6.2]{schildkraut2023}. 
 In this paper, we prove that, under the assumption that the algebraic degree of $\lambda$ is $2$ or $\lambda\leq 2$,  Condition \ref{condition} is sufficient for $\kappa(\lambda)<\infty$.
 
\begin{theorem}\label{thm:main1}
Suppose $\lambda\in (0,\infty)$ satisfies one of the following properties:
\begin{itemize}
    \item [(i)] $\lambda$ has algebraic degree $2$;
    \item [(ii)] \(\lambda\leq 2\).
\end{itemize}
Then $\kappa(\lambda)<\infty$ if and only if $\lambda$ satisfies  Condition \ref{condition}.
\end{theorem}

Furthermore, we derive precise values of the spectral radius order of two infinite families of quadratic algebraic integers. 

\begin{theorem}\label{thm:compute}
  Let $n,m\in \mathbb{Z}^+$. We have the following facts:
\begin{itemize}   
    \item [(i)]   If $n> \frac{(m-1)^2}{4}$, then $\kappa\left(\sqrt{n(n+m)}\right)=2n+m$;
  \item [(ii)] If $n>(m-1)^2+1$, then 
  $$\kappa\left(\frac{n-m-1+\sqrt{(n-m-1)^2+4m(n-m)}}{2}\right)=n.$$
\end{itemize}
\end{theorem}

It is worth noting that $\kappa(\lambda)$ is finite if and only if $\lambda$ can be realized as the largest eigenvalue of a graph. Problems of similar spirit have been investigated by several authors. In \cite{MR1189507}, Estes proved that every totally real algebraic integer can be realiezed as an eigenvalue of a graph. Salez \cite{MR3315609} strengthened this result by showing that every totally real algebraic integer is an eigenvalue of a tree graph. In \cite{MR766106}, Lind completely characterized the set of numbers given by the spectral radii of directed graphs.

This paper is structured as follows. In Section \ref{sec:pre}, we recall some basics from spectral graph theory. In Section \ref{sec:BasicResult}, we show several basic properties of the spectral radius order. We provide the proofs for Theorem \ref{thm:main1} and Theorem \ref{thm:compute} in Section \ref{sec:main1} and Section \ref{sec:calculation}, respectively.
}
\section{Preliminaries}\label{sec:pre}
{
In this paper, we only consider finite simple graphs, i.e., graphs without self-loops and multiple edges. For any given graph \( G = (V, E) \), we label the vertex set \( V \) as \( \{1, 2, \ldots, |V|\} \) for the sake of notational convenience. For any $i,j\in V$, we use $i\sim j$ to denote $\{i,j\}\in E$. We use  $K_n$ to denote the $n$-vertex complete graph and use $K_{s,t}$ to denote the complete bipartite graph with bipartite sets of size $s$ and $t$.

The complement graph $\overline{G}$ of a graph $G=(V,E)$ is the graph whose vertex set is $V$ and two vertices $i$ and $j$ are adjacent in $\overline{G}$ if and only if $i$ is not adjacent to $j$ in $G.$

Given two graphs $G_1=(V_1,E_1)$ and $ G_2=(V_2,E_2)$, we denote by $G_1\vee G_2$ the join of two graphs $G_1$ and $ G_2$, i.e.,  $G_1\vee G_2=(V,E)$ with $V=V_1\cup V_2$ and 
$$E:=E_1\cup E_2\cup\left\{ \{i,j\}:i\in V_1\text{ and }j\in V_2\right\}.$$

The Cartesian product of graphs  $G_1\text{ and } G_2$ is the  graph $G=(V_1\times V_2,E)$ with $(i,j)\sim (i',j')$ if and only if either $i=i'$ and $j\sim j'$ or $i\sim i'$ and $j= j'$}.

The Kronecker product of two graphs  $G_1\text{ and } G_2$ is the  graph $G=(V_1\times V_2,E)$ with $(i,j)\sim (i',j')$ if and only if  $i\sim i'$ and $j\sim j'$.

Let \( G\) be a graph. The adjacency matrix of \( G \) is a \( |V| \times |V| \) matrix \( A_G = (a_{ij})_{1 \leq i, j \leq |V|} \), where \( a_{ij} \) is defined as follows:
\[
a_{ij} := \begin{cases} 
0, & \text{if } i \not\sim j, \\
1, & \text{if } i \sim j.
\end{cases}
\]
   We list all eigenvalues of $A_G$  with multiplicity as  $$\lambda_1(G)\geq \lambda_2(G)\geq \cdots\geq \lambda_{|V|}(G).$$
   By the Perron-Frobenius theorem, 
   we have $\lambda_1(G)\geq |\lambda_i(G)|$ for any $i=1,\ldots,|V|$.
 We call $\lambda_1(G)$ the spectral radius of $G$. In general, for a non-negative matrix $M$, we call its largest positive eigenvalue as the spectral radius of $M$ and denote it as $\lambda_1(M).$

We recall the following lemma, which is a direct consequence of the Perron-Frobenius theorem, see, e.g.,  \cite{MR347860}.
\begin{lemma}\label{lemma:pf1}
  Let $H$ be a subgraph of a connected graph $G$. We have $\lambda_1(H)\leq\lambda_1(G)$, where the equality holds if and only if $H=G$.
\end{lemma}
   
The degree matrix of \( G \) is a \( |V| \times |V| \) diagonal matrix \( D_G \) whose \( i \)-th diagonal entry is the degree of vertex \( i \). The signless Laplacian of \( G \) is defined as $A_G + D_G$.

We denote the set of all $\lambda\in[0,+\infty)$ with $\kappa(\lambda)<\infty$ by  $\mathbb{S}$, that is,
     $$\mathbb{S}:=\{\lambda\in [0,+\infty)\ | \ \kappa(\lambda)<+\infty\}.$$

Let $\gamma$ be an algebraic integer and $p(x)$ be its minimal polynomial. We call a root of $p(x)$ an algebraic conjugate of $\gamma$. If all algebraic conjugates of $\gamma$ are real, we say $\gamma$ is  totally real.

\section{Basic properties of the spectral radius order}
\label{sec:BasicResult}
 In this section, we derive some basic properties of the spectral radius order. 
\begin{proposition}\label{pro:semi}
    The set $\mathbb{S}$ forms a semiring under real number addition and multiplication.
\end{proposition}
\begin{proof}
    Since $\kappa(0)=1<\infty$ and $\kappa(1)=2<\infty$, we have $0,1\in  \mathbb{S}$. To prove this proposition, it suffices to prove that the $\mathbb{S}$ is closed under real number addition and multiplication. 

    Let $\lambda,\lambda'\in \mathbb{S}$. By definition of spectral radius order, there exist two graphs $G_1$ and $G_2$ such that $\lambda_1(G_1)=\lambda$ and $\lambda_1(G_2)=\lambda'$.   By direct computation, we have that the spectral radius of the Cartesian product (resp., Kronecker product) of $G_1$ and $G_2$ is $\lambda+\lambda'$ (resp., $\lambda\lambda'$). This implies that $\mathbb{S}$ is closed under real number addition and multiplication.
\end{proof}
It is worth noting that  $\mathbb{S}$ is not closed under real number division or 
 subtraction in general, since \( \lambda \) cannot be the spectral radius of any graph if \( \lambda \in (0, 1) \).  

The next proposition tells us that the spectral radius of the signless Laplacian of any graph has a finite spectral radius order.

\begin{proposition}
    For a graph \( G \), let \( \mu \) be the spectral radius of the signless Laplacian of \( G \). Then $\kappa(\mu)<\infty.$
\end{proposition}
\begin{proof}
 Let $L(G)$ be the line graph of \( G \). We have  \( \lambda_1(L(G)) = \mu - 2 \) by \cite[Proposition 1.4.1]{MR2882891}. Since $2\in \mathbb{S}$, we have  $\mu\in \mathbb{S}$ by Proposition \ref{pro:semi}.
\end{proof}

It is interesting to ask whether the spectral radius of the Laplacian of any graph has finite spectral radius order or not. 

Although it is difficult to determine all elements of $\mathbb{S}$, the limit points of $\mathbb{S}$ can be clarified by the works of Hoffman \cite{MR347860} and Shearer \cite{MR986863}.
\begin{theorem}[\cite{MR347860}]
For \( n \in \mathbb{Z}^+ \), let \( \beta_n \) be the only positive root of the polynomial
\[
P_n(x) = x^{n + 1} - (1 + x + x^2 + \cdots + x^{n - 1}).
\]
Let \( \alpha_n = \beta_n^{1/2} + \beta_n^{-1/2} \). Then
\[
2 = \alpha_1 < \alpha_2 < \cdots
\]
are all limit points of $\mathbb{S}$ smaller than $ \sqrt{2+\sqrt{5}}$.
\end{theorem} 
\begin{theorem}[\cite{MR986863}]
\label{shearer}
     For any $\lambda\geq \sqrt{2+\sqrt{5}}$, there exists a sequence of graphs $G_1,\ G_2,\ldots,$ such that $G_i$ is a proper subgraph of $G_{i+1}$ and $\Mylim{i\rightarrow\infty}{\lambda_1(G_i)}=\lambda.$ 
\end{theorem}
\begin{remark}
  By Theorem \ref{shearer}, the set of limit points of $\mathbb{S}$ greater than $\sqrt{2+\sqrt{5}}$ is $$\left[\sqrt{2+\sqrt{5}},+\infty\right).$$
    In fact, Shearer \cite{MR986863} proved it by constructing such a sequence of graphs $G_1,G_2,\ldots$. We point out that all the graphs in his construction are connected.
    \end{remark}
As an application of Theorem \ref{shearer}, we derive the following corollary.
\begin{corollary}
    Given two numbers $\sqrt{2+\sqrt{5}}\leq a<b,$ we have $\sup\{\kappa(\lambda)|\ \lambda\in\mathbb{S}\cap(a,b) \}=+\infty.$
\end{corollary}
\begin{proof}
    Take $\lambda\in(a,b).$ By Theorem \ref{shearer}, there exists a sequence of connected graphs $\{G_i\}_{i=1}^{\infty}$ such that $\Mylim{i\rightarrow\infty}{\lambda_1(G_i)}=\lambda$ and $G_i$ is a proper subgraph of $G_{i+1}.$ Without loss of generality, we assume $\lambda_1(G_i)\in(a,b)$ for any $i\geq 1$. By Lemma \ref{lemma:pf1}, we have $\lambda_1(G_i)<\lambda_1(G_{i+1})$. Thus, there exist infinitely many distinct numbers in $(a,b)$ whose spectral radius orders are finite. However, for any $m\in\mathbb{Z}^+,$ there are only finitely many graphs whose orders are smaller than $m.$ This implies that there are only finitely many distinct numbers with spectral radius orders smaller than $m$. We conclude that $\Mylim{i\rightarrow\infty}{\kappa(\lambda_1(G_i))}=+\infty.$ This completes the proof.
\end{proof}
    
\section{Proof of Theorem \ref{thm:main1}}

\label{sec:main1}
In this section, we show the necessary Condition \ref{condition} is also sufficient for a real number $\lambda$ to have finite spectral radius order in the following two cases:
\begin{enumerate}
    \item [(i)] $\lambda$ is a quadratic algebraic integer;
    \item  [(ii)] $\lambda\leq 2$.
\end{enumerate}
We provide the proofs for the above two cases in the following two subsections, respectively. Indeed, Theorem \ref{thm:main1} follows from Theorem \ref{theorem:realizeA} and Theorem \ref{thm:small2} below.

\subsection{Quadratic algebraic integers}
 Let us first prepare the concepts and results needed in our proof.  

    \begin{definition}[{\cite[Section 2.3]{MR2882891}}]
        \label{QuotientMatrix}
        Suppose \( M=(m_{ab}) \) is a complex matrix whose rows and columns are indexed by \( \{1, \ldots, n\} \). Let \( \{X_1, \ldots, X_m\} \) be a partition of the index set \( \{1, \ldots, n\} \). We partition the matrix \( M \)  according to \( \{X_1, \ldots, X_m\} \), i.e.,
\[
M=\left[
\begin{matrix}
      M_{1,1} & \ldots & M_{1,m} \\
            \vdots & & \vdots \\
            M_{m,1} & \ldots & M_{m,m}
\end{matrix}
\right],
\]
where \( M_{i,j} \) denotes the submatrix of \( M \) formed by the rows with indices in \( X_i \) and the columns with indices in \( X_j \). Then the matrix \( B = (b_{ij}) \), where
 \[b_{ij}:=\frac{1}{|X_i|}\sum_{a\in X_i}\sum_{b\in X_j}m_{ab}, \,\,i,j=1,\ldots,m,\] 
is called the \emph{quotient matrix} of \( M \) with respect to the partition \(\{X_1,\ldots, X_m\}\).
         If  $M_{i,j}$ has a constant row sum, then $B$ is called an \emph{equitable} quotient matrix. 
    \end{definition}
   
     The subsequent theorem establishes a relationship between the spectral radii of a non-negative matrix $M$
 and its equitable quotient matrices.
    \begin{theorem}[{\cite[Theorem 2.5]{YOU201921}}]
        \label{SpectralRadiusRelation}
        Let $M$ be a non-negative matrix and $B$ be an equitable quotient matrix of $M$. Then $\lambda_1(B)=\lambda_1(M).$
    \end{theorem}

Furthermore, we need the following lemma, which is a folklore result.

 \begin{lemma}[Existence of regular graphs]\label{lemma:regular} Let $n,k\in \mathbb{Z}^+$. There exists a \( k \)-regular graph of order \( n \) if and only if \( n \geq k + 1 \) and that \( nk \) is even.  
  \end{lemma}

    Let $\lambda$ be a quadratic algebraic integer. That is, $\lambda$ is a root of a quadratic polynomial $x^2+bx+c,$ with $b,c\in \mathbb{Z}$. If $\lambda$ satisfies Condition \ref{condition}, then we must have $b\leq 0$ and $c\leq\left\lfloor \frac{b^2}{4}\right\rfloor.$  Next, we show that Condition \ref{condition} are sufficient for quadratic algebraic integers to have finite spectral radius orders.

    \begin{theorem}
    \label{theorem:realizeA}
        Given any quadratic polynomial $x^2+bx+c$ with $b\leq 0$ and $c\leq \left\lfloor{\frac{b^2}{4}}\right\rfloor.$ Let $\lambda$ be the largest root of this quadratic polynomial. Then $\lambda$ can be realized as the spectral radius of a graph.
    \end{theorem}
    \begin{proof}
    We divide the proof into two cases. Let us denote $m:=-\left\lfloor \frac{b}{2} \right\rfloor$.
 \\
 {\textbf{Case 1:}} $c=\left\lfloor \frac{b^2}{4} \right\rfloor$. Then $\lambda=m=\lambda_1(K_{m+1})$.
 \\
{\textbf{Case 2:}} $c<\left\lfloor \frac{b^2}{4} \right\rfloor$. 
If \(b = -2m\), we define  
    \[
    A: =(a_{ij})= \begin{bmatrix} 
    m & 1 \\ 
    m^2 - c & m 
    \end{bmatrix}.
    \]  
Otherwise, if \(b=-2m+1\), we define 
 \[
    A:=(a_{ij}) = \begin{bmatrix} 
    m & 1 \\ 
    m^2 - m - c & m - 1 
    \end{bmatrix}.
    \]  
    Notice that the polynomial $x^2+bx+c$ is the characteristic polynomial of $A$. Let $M$ be an \emph{even} integer larger than $m$. Set $n_1:=Ma_{21}$ and $n_2:=Ma_{12}$. By Lemma \ref{lemma:regular}, there exist a $a_{11}$-regular graph $G_1$ and $a_{22}$-regular graph $G_2$ of order $n_1$ and $n_2$, respectively. Denote the vertex set of $G_i$ by $V_{G_i}:=\{v_{ij}\ \vert \ j=1,\ldots,n_i\}$, $i=1,2$.   We construct a graph \( G \) with vertex set $V_{G_1}\cup V_{G_2}$ as follows. Two vertices $v_{ik},\ v_{jl}$ are adjacent if and only if
        \begin{center}
            either $i=j$ and $v_{ik}\sim v_{jl}$  in $G_i$, or $i\neq j$ and $l=\left(\left\lceil \frac{k}{a_{ji}}\right\rceil-1\right) a_{ij}+b$ for some $1\leq b\leq a_{ij}$.
        \end{center}
 By construction, the matrix $A$ is the equitable quotient matrix of the adjacency matrix $A_G$, with respect to the partition $\{V_{G_1},V_{G_2} \}$. By Theorem \ref{SpectralRadiusRelation}, it holds that $\lambda_1(A)=\lambda_1(G).$

This completes the proof.
    \end{proof}

We provide an example to illustrate the proof of Theorem \ref{theorem:realizeA}.
    \begin{example}\label{example:figure}
        Consider the quadratic polynomial $x^2-2x-2$. Its largest positive root is $1+\sqrt{3}.$ 
        This is the case $b=-2$ and $c=-2<\left\lfloor\frac{b^2}{4}\right\rfloor$ in the proof of Theorem \ref{theorem:realizeA}. We consider the matrix
    \[A=\left[
        \begin{array}{cc}
            1 & 1\\
            3 & 1
        \end{array}
    \right].\]
Set $M=2,\ n_1=6,\ n_2=2$.
Then we can take $G_1$ consisting of 3 copies of $K_2$ and $G_2=K_2$. We construct a new graph $G$ from $G_1$ and $G_2$ according to the strategy described in the proof of Theorem \ref{theorem:realizeA}. We illustrate this construction in Figure \ref{ConstructionOfG}. In conclusion, we obtain a graph $G$ satisfying $\lambda_1(A)=\lambda_1(G)=1+\sqrt{3}.$\\
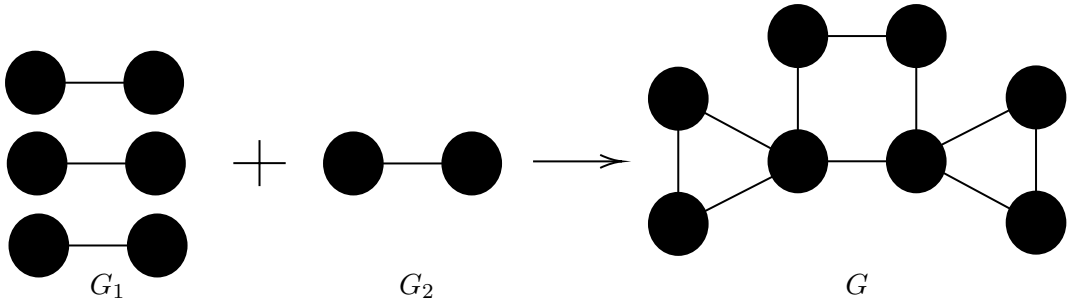
\begin{figure}[H]
    \centering
    \tikzset{every picture/.style={line width=0.75pt}} 

\begin{tikzpicture}[x=0.75pt,y=0.75pt,yscale=-1,xscale=1]

\draw  [fill={rgb, 255:red, 0; green, 0; blue, 0 }  ,fill opacity=1 ] (233,172.68) .. controls (233,164.07) and (239.55,157.08) .. (247.64,157.08) .. controls (255.72,157.08) and (262.27,164.07) .. (262.27,172.68) .. controls (262.27,181.3) and (255.72,188.28) .. (247.64,188.28) .. controls (239.55,188.28) and (233,181.3) .. (233,172.68) -- cycle ;
\draw  [fill={rgb, 255:red, 0; green, 0; blue, 0 }  ,fill opacity=1 ] (292.1,172.68) .. controls (292.1,164.07) and (298.65,157.08) .. (306.74,157.08) .. controls (314.82,157.08) and (321.37,164.07) .. (321.37,172.68) .. controls (321.37,181.3) and (314.82,188.28) .. (306.74,188.28) .. controls (298.65,188.28) and (292.1,181.3) .. (292.1,172.68) -- cycle ;
\draw [fill={rgb, 255:red, 0; green, 0; blue, 0 }  ,fill opacity=1 ]   (262.27,172.68) -- (292.1,172.68) ;
\draw  [fill={rgb, 255:red, 0; green, 0; blue, 0 }  ,fill opacity=1 ] (75.14,172.68) .. controls (75.14,164.07) and (81.69,157.08) .. (89.78,157.08) .. controls (97.86,157.08) and (104.41,164.07) .. (104.41,172.68) .. controls (104.41,181.3) and (97.86,188.28) .. (89.78,188.28) .. controls (81.69,188.28) and (75.14,181.3) .. (75.14,172.68) -- cycle ;
\draw  [fill={rgb, 255:red, 0; green, 0; blue, 0 }  ,fill opacity=1 ] (134.24,172.68) .. controls (134.24,164.07) and (140.79,157.08) .. (148.88,157.08) .. controls (156.96,157.08) and (163.51,164.07) .. (163.51,172.68) .. controls (163.51,181.3) and (156.96,188.28) .. (148.88,188.28) .. controls (140.79,188.28) and (134.24,181.3) .. (134.24,172.68) -- cycle ;
\draw [fill={rgb, 255:red, 0; green, 0; blue, 0 }  ,fill opacity=1 ]   (104.41,172.68) -- (134.24,172.68) ;

\draw  [fill={rgb, 255:red, 0; green, 0; blue, 0 }  ,fill opacity=1 ] (76.01,213.94) .. controls (76.01,205.32) and (82.56,198.34) .. (90.65,198.34) .. controls (98.73,198.34) and (105.28,205.32) .. (105.28,213.94) .. controls (105.28,222.55) and (98.73,229.54) .. (90.65,229.54) .. controls (82.56,229.54) and (76.01,222.55) .. (76.01,213.94) -- cycle ;
\draw  [fill={rgb, 255:red, 0; green, 0; blue, 0 }  ,fill opacity=1 ] (135.11,213.94) .. controls (135.11,205.32) and (141.66,198.34) .. (149.75,198.34) .. controls (157.83,198.34) and (164.38,205.32) .. (164.38,213.94) .. controls (164.38,222.55) and (157.83,229.54) .. (149.75,229.54) .. controls (141.66,229.54) and (135.11,222.55) .. (135.11,213.94) -- cycle ;
\draw [fill={rgb, 255:red, 0; green, 0; blue, 0 }  ,fill opacity=1 ]   (105.28,213.94) -- (135.11,213.94) ;

\draw  [fill={rgb, 255:red, 0; green, 0; blue, 0 }  ,fill opacity=1 ] (74.27,132.06) .. controls (74.27,123.44) and (80.82,116.46) .. (88.9,116.46) .. controls (96.99,116.46) and (103.54,123.44) .. (103.54,132.06) .. controls (103.54,140.67) and (96.99,147.65) .. (88.9,147.65) .. controls (80.82,147.65) and (74.27,140.67) .. (74.27,132.06) -- cycle ;
\draw  [fill={rgb, 255:red, 0; green, 0; blue, 0 }  ,fill opacity=1 ] (133.37,132.06) .. controls (133.37,123.44) and (139.92,116.46) .. (148.01,116.46) .. controls (156.09,116.46) and (162.64,123.44) .. (162.64,132.06) .. controls (162.64,140.67) and (156.09,147.65) .. (148.01,147.65) .. controls (139.92,147.65) and (133.37,140.67) .. (133.37,132.06) -- cycle ;
\draw [fill={rgb, 255:red, 0; green, 0; blue, 0 }  ,fill opacity=1 ]   (103.54,132.06) -- (133.37,132.06) ;

\draw  [fill={rgb, 255:red, 0; green, 0; blue, 0 }  ,fill opacity=1 ] (409.83,124.49) .. controls (417.92,124.49) and (424.47,131.48) .. (424.47,140.09) .. controls (424.47,148.71) and (417.92,155.69) .. (409.83,155.69) .. controls (401.75,155.69) and (395.2,148.71) .. (395.2,140.09) .. controls (395.2,131.48) and (401.75,124.49) .. (409.83,124.49) -- cycle ;
\draw  [fill={rgb, 255:red, 0; green, 0; blue, 0 }  ,fill opacity=1 ] (409.83,187.48) .. controls (417.92,187.48) and (424.47,194.47) .. (424.47,203.08) .. controls (424.47,211.7) and (417.92,218.68) .. (409.83,218.68) .. controls (401.75,218.68) and (395.2,211.7) .. (395.2,203.08) .. controls (395.2,194.47) and (401.75,187.48) .. (409.83,187.48) -- cycle ;
\draw [fill={rgb, 255:red, 0; green, 0; blue, 0 }  ,fill opacity=1 ]   (409.83,155.69) -- (409.83,187.48) ;

\draw  [fill={rgb, 255:red, 0; green, 0; blue, 0 }  ,fill opacity=1 ] (454.89,171.59) .. controls (454.89,162.97) and (461.44,155.99) .. (469.53,155.99) .. controls (477.61,155.99) and (484.16,162.97) .. (484.16,171.59) .. controls (484.16,180.2) and (477.61,187.19) .. (469.53,187.19) .. controls (461.44,187.19) and (454.89,180.2) .. (454.89,171.59) -- cycle ;
\draw  [fill={rgb, 255:red, 0; green, 0; blue, 0 }  ,fill opacity=1 ] (513.99,171.59) .. controls (513.99,162.97) and (520.54,155.99) .. (528.63,155.99) .. controls (536.71,155.99) and (543.26,162.97) .. (543.26,171.59) .. controls (543.26,180.2) and (536.71,187.19) .. (528.63,187.19) .. controls (520.54,187.19) and (513.99,180.2) .. (513.99,171.59) -- cycle ;
\draw [fill={rgb, 255:red, 0; green, 0; blue, 0 }  ,fill opacity=1 ]   (484.16,171.59) -- (513.99,171.59) ;

\draw  [fill={rgb, 255:red, 0; green, 0; blue, 0 }  ,fill opacity=1 ] (454.89,108.6) .. controls (454.89,99.98) and (461.44,93) .. (469.53,93) .. controls (477.61,93) and (484.16,99.98) .. (484.16,108.6) .. controls (484.16,117.21) and (477.61,124.2) .. (469.53,124.2) .. controls (461.44,124.2) and (454.89,117.21) .. (454.89,108.6) -- cycle ;
\draw  [fill={rgb, 255:red, 0; green, 0; blue, 0 }  ,fill opacity=1 ] (513.99,108.6) .. controls (513.99,99.98) and (520.54,93) .. (528.63,93) .. controls (536.71,93) and (543.26,99.98) .. (543.26,108.6) .. controls (543.26,117.21) and (536.71,124.2) .. (528.63,124.2) .. controls (520.54,124.2) and (513.99,117.21) .. (513.99,108.6) -- cycle ;
\draw [fill={rgb, 255:red, 0; green, 0; blue, 0 }  ,fill opacity=1 ]   (484.16,108.6) -- (513.99,108.6) ;

\draw [fill={rgb, 255:red, 0; green, 0; blue, 0 }  ,fill opacity=1 ]   (469.53,124.2) -- (469.53,155.99) ;
\draw [fill={rgb, 255:red, 0; green, 0; blue, 0 }  ,fill opacity=1 ]   (528.63,124.2) -- (528.63,155.99) ;
\draw [fill={rgb, 255:red, 0; green, 0; blue, 0 }  ,fill opacity=1 ]   (422.79,147.09) -- (456.48,165.35) ;
\draw [fill={rgb, 255:red, 0; green, 0; blue, 0 }  ,fill opacity=1 ]   (422.79,196.85) -- (457.07,179.21) ;
\draw  [fill={rgb, 255:red, 0; green, 0; blue, 0 }  ,fill opacity=1 ] (588.41,218.05) .. controls (580.33,218.05) and (573.77,211.07) .. (573.77,202.45) .. controls (573.77,193.84) and (580.33,186.85) .. (588.41,186.85) .. controls (596.49,186.85) and (603.05,193.84) .. (603.05,202.45) .. controls (603.05,211.07) and (596.49,218.05) .. (588.41,218.05) -- cycle ;
\draw  [fill={rgb, 255:red, 0; green, 0; blue, 0 }  ,fill opacity=1 ] (588.41,155.06) .. controls (580.33,155.06) and (573.77,148.08) .. (573.77,139.46) .. controls (573.77,130.85) and (580.33,123.86) .. (588.41,123.86) .. controls (596.49,123.86) and (603.05,130.85) .. (603.05,139.46) .. controls (603.05,148.08) and (596.49,155.06) .. (588.41,155.06) -- cycle ;
\draw [fill={rgb, 255:red, 0; green, 0; blue, 0 }  ,fill opacity=1 ]   (588.41,186.85) -- (588.41,155.06) ;

\draw [fill={rgb, 255:red, 0; green, 0; blue, 0 }  ,fill opacity=1 ]   (575.46,195.46) -- (541.77,177.19) ;
\draw [fill={rgb, 255:red, 0; green, 0; blue, 0 }  ,fill opacity=1 ]   (575.46,145.7) -- (541.18,163.33) ;

\draw  [fill={rgb, 255:red, 0; green, 0; blue, 0 }  ,fill opacity=1 ] (187.74,172.52) -- (213.87,172.52)(200.8,160.95) -- (200.8,184.08) ;
\draw [fill={rgb, 255:red, 0; green, 0; blue, 0 }  ,fill opacity=1 ]   (337.35,171.64) -- (379.78,171.64) ;
\draw [shift={(381.78,171.64)}, rotate = 180] [color={rgb, 255:red, 0; green, 0; blue, 0 }  ][line width=0.75]    (10.93,-3.29) .. controls (6.95,-1.4) and (3.31,-0.3) .. (0,0) .. controls (3.31,0.3) and (6.95,1.4) .. (10.93,3.29)   ;

\draw (491.83,226.75) node [anchor=north west][inner sep=0.75pt]   [align=left] {$\displaystyle G$};
\draw (114.59,226.75) node [anchor=north west][inner sep=0.75pt]   [align=left] {$\displaystyle G_{1}$};
\draw (269.08,226.75) node [anchor=north west][inner sep=0.75pt]   [align=left] {$\displaystyle G_{2}$};

\end{tikzpicture}
    \caption{The construction of $G$ whose spectral radius is $1+\sqrt{3}$}
    \label{ConstructionOfG}
\end{figure}
    \end{example}

\subsection{Numbers no larger than $2$}

The next theorem provides a characterization of algebraic integers, whose algebraic conjugates all lie in the interval $[-2,2].$ This theorem is due to Kronecker.

    \begin{theorem}[\cite{MR1578994}]\label{thm:Kronecker}
        Let \( f(x) \) be a monic polynomial with integer coefficients, whose roots are all real. If every root of \( f(x) \) lies in \([-2, 2]\), then these roots must take the form \( 2\cos\frac{2\pi k}{n} \), where \( k \) and \( n \) are integers.  
    \end{theorem}

   For $n\geq 2$, let $\Psi_n(x)$ be the monic minimal polynomial of $\cos\frac{2\pi}{n}.$  The next lemma gives all roots of $\Psi_n(x)$.
    \begin{lemma}[\cite{MR1215534}]\label{lem:Psin} For $n\geq 2,$ the roots of $\Psi_n(x)$ are $\cos\frac{2\pi k}{n},$ where $1\leq k\leq \frac{n}{2}$ and $(k,n)=1.$
    \end{lemma}
In particular, the above lemma tells that the polynomial $\Psi_n(x)$ has $\frac{\phi(n)}{2}$ roots, where $\phi(n)$ is Euler's totient function.

       Next, we prove the following theorem.
    \begin{theorem}\label{thm:small2}
        Any real number $\lambda\leq 2$  satisfying Condition \ref{condition} has finite spectral radius order.
    \end{theorem}
    \begin{proof}
If $\lambda=2$, we have $\lambda=\lambda_1(K_3)$. 
  
When $\lambda<2$, we claim that $\lambda=2\cos\frac{\pi}{n}=\lambda_1(P_{n-1})$ for some $n\geq 2$, where $P_{n-1}$ is a path of order $n-1$.

Indeed, by Condition \ref{condition}, all algebraic conjugates of $\lambda$ lie in the interval $[-2,2]$.
Applying Theorem \ref{thm:Kronecker}, we have $\lambda=2\cos\frac{2\pi k}{n}$ for some integers $k$ and $n$. Without loss of generality, we assume $(k,n)=1$. We divide the proof of the claim into two cases.

{\textbf{Case 1:}} $1\leq k\leq \frac{n}{2}$.

We derive from Lemma \ref{lem:Psin} that the monic minimal polynomial of $\lambda$ is $2^{\frac{\phi(n)}{2}}\Psi_n(\frac{x}{2})$. The largest root of $2^{\frac{\phi(n)}{2}}\Psi_n(\frac{x}{2})$ is $2\cos\frac{2\pi}{n}$.
Since $\lambda$ is the largest root of $2^{\frac{\phi(n)}{2}}\Psi_n(\frac{x}{2})$ by Condition \ref{condition} (2), we obtain that $\lambda=2\cos\frac{2\pi}{n}$ for some integer $n\geq 4$. Next we show $n$ has to be even.  Assume, otherwise, $n=2k+1\ (k\geq 1)$. Then $\left\lfloor\frac{n}{2}\right\rfloor=k$ is coprime with $n.$ Thus, we derive from Lemma \ref{lem:Psin} that $2\cos\frac{2\pi \left\lfloor\frac{n}{2}\right\rfloor}{n}=-2\cos\frac{\pi}{n}$ is a root of $2^{\frac{\phi(n)}{2}}\Psi_n(\frac{x}{2})$. Notice that its absolute value is strictly larger than $\lambda$. This violates Condition \ref{condition} (2). Hence, $n$ is an even integer and $\lambda=2\cos\frac{\pi}{k}$, where $n=2k$.

{\textbf{Case 2:}} $\frac{n}{2}<k< n$.

In this case,  we have 
    \begin{equation*}
        \lambda=2\cos\frac{2\pi k}{n}=-2\cos\frac{2\pi(2k-n)}{2n}=-2\cos\frac{2\pi p}{q},
    \end{equation*}
    where $1\leq p\leq \frac{q}{2},\ (p,q)=1$ and $\frac{p}{q}=\frac{2k-n}{2n}.$
By Lemma \ref{lem:Psin}, the monic minimal polynomial of $\lambda$ is $(-2)^{\frac{\phi(q)}{2}}\Psi_q(-\frac{x}{2}).$ When $q=2\ell+1, \ell\geq 1$, the largest root of $(-2)^{\frac{\phi(q)}{2}}\Psi_q(-\frac{x}{2})$ is $-2\cos\frac{2\pi \ell}{q}=2\cos\frac{\pi}{q}.$ Hence, we have $\lambda=2\cos\frac{\pi}{q}$ by Condition \ref{condition}. When $q=2\ell,\ell \geq 2$, the largest root of $(-2)^{\frac{\phi(q)}{2}}\Psi_q(-\frac{x}{2})$ is no more than $-2\cos\frac{2\pi(\ell-1)}{q}=2\cos\frac{\pi}{\ell}.$ Observe that $-2\cos\frac{2\pi}{q}=-2\cos\frac{\pi}{\ell}$ is a root of the polynomial $(-2)^{\frac{\phi(q)}{2}}\Psi_q(-\frac{x}{2})$. Since $\lambda$ satisfies Condition \ref{condition}, the only possibility is $\lambda=2\cos\frac{\pi}{\ell}.$ This completes the proof. 
    \end{proof}

\section{Proof of Theorem \ref{thm:compute}}
\label{sec:calculation}
In this section, we prove Theorem \ref{thm:compute}. 
We first prove the following lemma. 
\begin{lemma}\label{lemma:bipartite}
    Let $G=(V,E)$ be a bipartite graph. Then, we have $$\lambda_1(G)\leq \frac{|V|}{2}.$$
\end{lemma}
\begin{proof}
    Assume the bipartite set of $G$ is $A$ and $B$. Let $|A|=s$ and $|B|=t$. Since $G$ is a subgraph of $K_{s,t}$, by Lemma \ref{lemma:pf1}, we have $$\lambda_1(G)\leq \lambda_1(K_{s,t})=\sqrt{st}\leq \frac{s+t}{2}.$$
\end{proof}

\begin{proof}[Proof of Theorem \ref{thm:compute} $(i)$]
Because $\lambda_1(K_{n,n+m})=\sqrt{n(n+m)}$, we have $$\kappa\left(\sqrt{(n(n+m)}\right)\leq 2n+m.$$ 

If $\kappa\left(\sqrt{(n(n+m)}\right)< 2n+m$,  then there exists a graph $G=(V,E)$ such that $$\lambda_1(G)=\sqrt{n(n+m)}\,\,\text{ and }\,\,|V|\leq 2n+m-1.$$ Since the monic minimum polynomial of $\sqrt{n(n+m)}$ is $x^2-n(n+m)$, we have $$x^2-n(n+m)\big|P_G(x),$$ where $P_G(x)$ is the characteristic polynomial of $A_G$. It follows that $-\sqrt{n(n+m)}$ is also an eigenvalue of $A_G$. This implies that $G$ is a bipartite graph. By Lemma \ref{lemma:bipartite}, we have $$\lambda_1(G)\leq \frac{2n+m-1}{2}<\sqrt{n(n+m)}, $$
   where the last inequality above holds by $n>\frac{(m-1)^2}{4}$. This is a contradiction. Therefore $$\kappa\left(\sqrt{n(n+m)}\right)=2n+m.$$ This concludes the proof.
\end{proof}
To prove Theorem \ref{thm:compute} $(ii)$, we need the following observation.
\begin{lemma}\label{lemma:KVK}
  For any $n> m\geq 2$, we have \[\lambda_1\left(K_{n-m}\vee \overline{K_{m}}\right)=\frac{n-m-1+\sqrt{(n-m-1)^2+4m(n-m)}}{2}.\]
\end{lemma}
\begin{proof}
   We consider a function $f$ on $K_{n-m}\vee \overline{K_{m}}$ with the following form:
  \begin{equation*}
          f(x)=
            \begin{cases}
              1, &  x\in K_{n-m}\\
              \frac{-(n-m-1)+\sqrt{(n-m-1)^2+4m(n-m)}}{2m}, &  x\in \overline{K_m}
            \end{cases}.
  \end{equation*}
  Then we have $A_Gf=\lambda f,$ where $\lambda=\frac{n-m-1+\sqrt{(n-m-1)^2+4m(n-m)}}{2}.$ Since $G$ is connected and $f$ is positive on all vertices of $G,$ by the Perron-Frobenius theorem, we obtain $\lambda=\lambda_1(K_{n-m}\vee \overline{K_{m}}).$
\end{proof}

\begin{proof}[Proof of Theorem \ref{thm:compute} $(ii)$]
   By Lemma \ref{lemma:KVK}, we have $\kappa\left(\frac{n-m-1+\sqrt{(n-m-1)^2+4m(n-m)}}{2}\right)\leq n$. For any graph $G=(V,E)$ with $|V|\leq n-1$, we have $$\lambda_1(G)\leq n-2< \frac{n-m-1+\sqrt{(n-m-1)^2+4m(n-m)}}{2},$$
   where the last inequality above is due to $n>(m-1)^2+1$. This concludes the proof.
\end{proof}

\section*{Acknowledgement}
This work is supported by the National Key R \& D Program of China 2023YFA1010200 and the National Natural Science Foundation of China No. 12431004.

\bibliography{bib/paper.bib}

\end{document}